\documentclass[reqno,11pt]{amsart}
\usepackage[foot]{amsaddr}
\usepackage{bbold}
\usepackage{charter}
\usepackage{float}
\usepackage{enumitem}
\usepackage[margin=1in]{geometry}
\usepackage[colorlinks=true,linkcolor=blue,citecolor=blue]{hyperref}
\usepackage{pgfplots}

\pgfplotsset{compat=newest,tick label style={font=\scriptsize}}

\setlist[enumerate,1]{label = \upshape(\roman*), ref = (\arabic*)}

\theoremstyle{plain}
\newtheorem{theorem}{Theorem}[section]
\newtheorem{prop}[theorem]{Proposition}
\newtheorem{lemma}[theorem]{Lemma}
\newtheorem{cor}[theorem]{Corollary}
\theoremstyle{definition}

\newcommand{\myeqref}[1]{Eqn.~\eqref{#1}}
\newcommand{\D}{\mathbb{D}}
\newcommand{\N}{\mathbb{N}}
\newcommand{\C}{\mathbb{C}}

\newcommand{\CO}{C_\varphi}
\newcommand{\CDO}{D_\varphi}
\newcommand{\DCO}{DC_\varphi}
\newcommand{\Ph}{\varphi}
\newcommand{\SI}{\sigma}

\title{Composition-Differentiation Operators on $S^2(\D)$}

\author{Robert F.~Allen\textsuperscript{1}, Katherine C.~Heller\textsuperscript{2}, and Matthew A.~Pons\textsuperscript{2}}
\address{\textsuperscript{1}Department of Mathematics and Statistics, University of Wisconsin-La Crosse}
\address{\textsuperscript{2}Department of Mathematics \& Actuarial Science, North Central College}


\begin{document}

\begin{abstract} 
	We investigate composition-differentiation operators acting on the space $S^2$, the space of analytic functions on the open unit disk whose first derivative is in $H^2$.  Specifically, we determine characterizations for bounded and compact composition-differentiation operators acting on $S^p$.  In addition, for particular classes of inducing maps, we compute the norm, and identify the spectrum.  Finally, for particular linear fractional inducing maps, we determine the adjoint of the composition-differentiation operator acting on weighted Bergman spaces which include $S^2, H^2$, and the Dirichlet space.
\end{abstract}	
	
\maketitle

\section{Introduction}

Let $\D=\{z\in\C:|z|<1\}$ denote the open unit disk in the complex plane and let $\varphi$ be an analytic self-map of the disk; that is, $\varphi(\mathbb{D})\subseteq\mathbb{D}$. For $\mathcal{X}$ a Hilbert or Banach space of analytic functions, define the composition operator $C_{\varphi}$ on $\mathcal{X}$ by $C_{\varphi}(f)=f\circ\varphi$ for $f\in\mathcal{X}$. Composition operators have been studied in great detail on spaces of analytic functions such as the Hardy space $H^2$, the weighted Bergman spaces $A^2_{\alpha}$, the Dirichlet space $\mathcal{D}$, the Bloch space $\mathcal{B}$ and many others. The interested reader can learn much about composition operators in \cite{CowenMacCluer:1995}.

Recently authors have begun to investigate the composition-differentiation operator $D_{\varphi}$ on $\mathcal{X}$, defined by $D_{\varphi}(f)=f'\circ\varphi$ for $f\in\mathcal{X}$ (see \cite{FatehiHammond:2020}, \cite{HibschweilerPortnoy:2005}, \cite{Ohno:2006}, \cite{Ohno:2009}). The primary goal of this paper is to investigate the adjoint of a bounded composition-differentiation operator acting on the space $S^2$ and compare this with known results (\cite{FatehiHammond:2020}, \cite{AllenHellerPons:2015}). To do so, we first  give conditions under which a composition-differentiation operator is bounded on $S^2$. In fact, we give a more general result for the $S^p$ spaces and also determine compactness criteria. This work also parallels the results of \cite{Ohno:2014}. Additionally, we explore the spectrum and norm of such an operator for a specific class of symbols $\Ph$. Finally, we explore the adjoint of a composition-differentiation operator acting on the standard weighted Bergman spaces $A_{\alpha}^2$ and $S^2$. We then compare these formulas and discuss how to tie these results together.

\section{Preliminaries}

The space $S^2$ consists of analytic functions in the unit disk whose derivative is in the Hardy space $H^2$.  In terms of a power series, if $f(z)=\sum_{n=0}^{\infty}a_n z^n$, then
\[
\|f\|_{S^2}^2=|a_0|^2+\sum_{n=1}^{\infty}|a_n|^2n^2<\infty.
\]
Thus $S^2$ is a weighted Hardy space with weight sequence $\beta(n)$ given by $\beta(0)=1$ and $\beta(n)=||z^n||_{S^2}=n$ for $n\geq 1$. We can also define the norm on $S^2$ in terms of an integral:
\[
||f||^2_{S^2}=|f(0)|^2+\int_0^{2\pi}|f'(e^{i\theta})|^2\frac{d\theta}{2\pi},
\] where the function in the integral is the radial limit function, which exists almost everywhere for $H^2$ functions.
It is straight forward to show that the two norms given are in fact equal. The inner product on $S^2$ can be recovered from the norm and is given by
\[
\langle f,g\rangle_{S^2}=f(0)\overline{g(0)}+\int_0^{2\pi}f'(e^{i\theta})\overline{g'(e^{i\theta})}\frac{d\theta}{2\pi}.
\]
We also note that the norm on $S^2$ can be written as 
\[
||f||^2_{S^2}=|f(0)|^2+||f'||^2_{H^2}.
\]

For each point $w\in\mathbb{D}$, evaluation of functions in $S^2$ at $w$ is a bounded linear functional and for each $f\in S^2$, $f(w)=\langle f,K_w\rangle$ where
\[
K_w(z)=1+\sum_{n=1}^{\infty}\frac{(\overline{w}z)^n}{n^2}.
\]
Additionally, for $w\in\mathbb{D}$, evaluation of the first derivative of functions in $S^2$ at $w$ is a bounded linear functional \cite[Theorem 1.9]{CowenMacCluer:1995} and $f'(w)=\langle f, K^{(1)}_w\rangle$, where
\begin{align}
K_w^{(1)}(z)&=\frac{1}{\overline{w}}\log\left(\frac{1}{1-\overline{w}z}\right)
\end{align}
for $w\neq 0$ and $K_0^{(1)}(z)=z$ with $||K_w^{(1)}||^2_{S^2}=\frac{1}{1-|w|^2}$. 

In a similar manner, for $p>1$ we can define the more general space $S^p$ as the set of analytic
functions $f$ on the unit disk whose derivative $f'$ lies in the Hardy space $H^p$ and define the norm on $S^p$ by

\[
||f||_{S^p}=|f(0)|+||f'||_{H^p}
\] where the $H^p$ norm is defined by 
\[
\|g\|_{H^p}^p=\int_0^{2\pi}|g(e^{i\theta})|^p\frac{d\theta}{2\pi}
\] and again we are thinking of the function in the integral as the radial limit function of $g\in H^p$.

When $p=2$, this definition gives the same set of functions as we previously defined for $S^2$, but with an equivalent norm and only Banach space structure. For the sake of notational ease, we will use this second form  of the norm to discuss boundedness and compactness so there is no need to consider a separate case for $S^2$.

\section{Boundedness and Compactness}

In order to discuss adjoints of composition-differentiation operators on $S^2$, we first need to understand when these operators are bounded. That is the focus of this section and we also discuss compactness.  As we will see, the class of bounded composition-differentiation operators on the $S^p$ spaces if somewhat limited. Our first result will be familiar to experts but we include a proof for the sake of completeness. In the following proof, we focus attention on the operator $\DCO$. The notation $\CDO=\CO D$ was adopted in \cite{FatehiHammond:2020} since the the differentiation operator is not bounded on $S^p$, but to avoid introducing any new notation we will use $\DCO$ instead of writing the operator as a single symbol.

\begin{prop}\label{prop:boundedinto}
Suppose $\mathcal{X}$ and $\mathcal{Y}$ are functional Banach spaces of analytic functions defined on $\D$.  Let $\Ph$ be a self-map of $\D$ and let $T:\mathcal{X}\to\mathcal{Y}$ be either $\CO$, $\CDO$, or $D\CO$. Then $T$ is bounded from $\mathcal{X}$ to $\mathcal{Y}$ if and only if $T$ maps $\mathcal{X}$ into $\mathcal{Y}$.
\end{prop}

\begin{proof}
Focusing on $\DCO$, necessity of the condition is clear and we will verify sufficiency. Suppose $\DCO$ maps $\mathcal{X}$ into $\mathcal{Y}$. To conclude that $\DCO$ is bounded, we will appeal to the Closed Graph Theorem.  Let $(f_n)$ be a sequence in $\mathcal{X}$ converging in norm to $f\in\mathcal{X}$.	Also suppose $(\DCO f_n)$ converges in norm to $g\in\mathcal{Y}$. For $x\in \D$, it is clear that \[(\DCO f_n)(x)=f_n'(\Ph(x))\Ph'(x)=K_{\Ph(x)}^{(1)}(f_n)\Ph'(x).\] Also, our assumptions on $(f_n)$ combined with the continuity of the evaluation functional guarantee that $\left(K_{\Ph(x)}^{(1)}(f_n)\right)$ converges to $K_{\Ph(x)}^{(1)}(f)$.  Thus, for each $x\in \D$, we have that $((\DCO f_n)(x))$ converges to \[K_{\Ph(x)}^{(1)}(f)\Ph'(x)=f'(\Ph(x))\Ph'(x)=(\DCO f)(x).\]

On the other hand, in a functional Banach space, norm convergence implies point-wise convergence and so, for $x\in\D$, we see that $((\DCO f_n)(x))$ converges to $g(x)$ and thus $(\DCO f)(x)=g(x)$ for every $x\in \D$ or $\DCO f=g$.  Therefore we conclude that $\DCO$ is bounded by the Closed Graph Theorem. We note that an analogous proof holds for the other two operators in the statement.
\end{proof}

Next we discuss boundedness of $\CDO$ acting on the $S^p$ spaces. Since the function $f(z)=z^2/2\in S^p$ for $p>1$ and $g(z)=z\in H^p$ for $p>1$, a necessary condition for each of the three operators in the next theorem to be bounded is that $\Ph\in S^p$. This fact follows from the previous theorem and we take it as a standing assumption on $\Ph$ throughout the remainder of this section.

\begin{theorem}\label{thm:boundedoperators}
Let $\Ph$ be an analytic self-map of $\D$ with $\Ph\in S^p$. Then the following are equivalent.

\begin{enumerate}
	
\item $\CDO:S^p\to S^p$ is bounded.

\item $\CO:H^p\to S^p$ is bounded.

\item $\DCO:H^p\to H^p$ is bounded.

\end{enumerate}
\end{theorem}

\begin{proof}
$(i)\Rightarrow(ii)$ Suppose $D_\Ph:S^p\to S^p$ is bounded. To show $\CO:H^p\to S^p$ is bounded, let $g\in H^p$. Then choose $f\in S^p$ with $f'=g$, i.e. let \[f(z)=\int_0^z g(w)\,dw.\] Then \[\CO g=g\circ \Ph=f'\circ \Ph=\CDO f\in S^p.\] Hence $\CO$ maps $H^p$ into $S^p$ and is bounded by Proposition \ref{prop:boundedinto}.

$(ii)\Rightarrow (iii)$ Suppose $\CO:H^p\to S^p$ is bounded. To show $D\CO:H^p\to H^p$ is bounded, let $f\in H^p$. Then $\CO f=f\circ\Ph\in S^p$, which implies $(f\circ\Ph)'\in H^p$, or $\DCO f\in H^p$. Therefore $\DCO$ maps $H^p$ into $H^p$ and is bounded by Proposition \ref{prop:boundedinto}.

$(iii)\Rightarrow (i)$ Suppose $\DCO:H^p\to H^p$ is bounded. To show $D_\Ph:S^p\to S^p$ is bounded, let $f\in S^p$.  Then $f'\in H^p$ and $\DCO f'=(f'\circ\Ph)'\in H^p$, which implies $f'\circ\Ph\in S^p$ or $\CDO f\in S^p$. Thus $\CDO$ maps $S^p$ into $S^p$ and is bounded by Proposition \ref{prop:boundedinto}.
\end{proof}

\begin{theorem}\label{thm:bounded}
Let $\Ph$ be an analytic self-map of $\D$ with $\Ph\in S^p$. Then $D_\Ph:S^p\to S^p$ is bounded if and only if $\|\Ph\|_{\infty}<1$.
\end{theorem}

\begin{proof}
The result follows from Theorem \ref{thm:boundedoperators} and \cite[Corollary 1]{HibschweilerPortnoy:2005}.
\end{proof}

\begin{cor}\label{cor:bounded}
Let $\Ph$ be an analytic self-map of $\D$ with $\Ph\in S^p$.  If $\CDO:S^p\to S^p$ is bounded, then $\CO:S^p\to S^p$ is bounded.
\end{cor}

Next we characterize compact $\CDO:S^p\to S^p$.  For composition operators, Shapiro showed that $\CO:S^p\to S^p$ is compact if and only if $\|\Ph\|_{\infty}<1$ in \cite[Theorem 2.1]{Shapiro}.  The boundary regularity of these spaces plays a critical role in this result. As will see below, this same result holds for composition-differentiation operators on $S^p$.

\begin{prop}[{\cite[Lemma 3.7]{Tjani:2003}}]\label{prop:compactreformulation}
Let $\Ph$ be a self-map of $\D$ and let $T:\mathcal{X}\to\mathcal{Y}$ be either $\CDO:S^p\to S^p$, $\CO:H^p\to S^p$, $D\CO:H^p\to H^p$. Then $T$ is compact if and only if whenever $(f_n)$ is a bounded sequence in $\mathcal{X}$ with $(f_n)\rightarrow 0$ uniformly on compact subsets of $\D$, it follows that $(\|Tf_n\|_{\mathcal{Y}})\rightarrow 0$.
\end{prop}

\begin{theorem}\label{thm:Supnormequivcompact}
Let $\Ph$ be self-map of $\D$ with $\Ph\in S^p$. Then $\CDO:S^p\to S^p$ is compact if and only if $\|\Ph\|_{\infty}<1$.
\end{theorem}

\begin{proof}
For $\Ph$ as given, suppose $\CDO:S^p\to S^p$ is compact. Then $\CDO$ is bounded and hence $\|\Ph\|_{\infty}<1$ by Theorem \ref{thm:bounded}.  Conversely, suppose $\|\Ph\|_{\infty}<1.$ Then there is an $0<M<1$ such that $\overline{\Ph(\D)}\subseteq M\overline{\D}$. We will appeal to Proposition \ref{prop:compactreformulation} and hence we let $(f_n)$ be a bounded sequence in $S^p$ with $(f_n)\to 0$ uniformly on compact subsets of $\D$. We will show $(\|\CDO f_n\|_{S^p})\to 0$ in order to conclude that $\CDO:S^p\to S^p$ is compact. Recall \[\|\CDO f_n\|_{S^p}=|f_n'(\Ph(0))|+\|(\CDO f_n)'\|_{H^p}.\] By our hypothesis on $(f_n)$, we know that $(f_n')\to 0$ uniformly on compact subsets of $\D$ by \cite[Theorem V.1.6]{Lang:1993} and it follows immediately that $(|f_n'(\Ph(0))|)\to0$. We are left to confirm that the quantities $\|(\CDO f_n)'\|_{H^p}$ tend to zero as $n\to \infty$. Appealing to \cite[Theorem V.1.6]{Lang:1993} again, we know that $(f_n'')\to 0$ uniformly on compact subsets of $\D$. Fixing $\varepsilon >0$, we choose $N\in \N$ with $|f_n''(z)|\leq \varepsilon/2\|\Ph'\|_{H^p}$ for $n\geq N$ and $z\in M\overline{\D}.$  Thus, if $n\geq N$ and $0<r<1$, we have \[\begin{aligned}\int_0^{2\pi}|f_n''(\Ph(re^{i\theta}))|^p|\Ph'(re^{i\theta})|^p\,\frac{d\theta}{2\pi}\leq \frac{\varepsilon^p}{2^p\|\Ph'\|_{H^p}^p}\int_0^{2\pi}|\Ph'(re^{i\theta})|^p\,\frac{d\theta}{2\pi}<\varepsilon^p.
\end{aligned}\] It follows that for $n\geq N$ we have \[\|(\CDO f_n)'\|_{H^p}^p=\sup_{0<r<1}\int_0^{2\pi}|f_n''(\Ph(re^{i\theta}))|^p|\Ph'(re^{i\theta})|^p\,\frac{d\theta}{2\pi}<\varepsilon^{p}\] and our conclusion follows.
\end{proof}

We immediately have the following corollary, which provides an interesting contrast to the parallel result for bounded $\CDO:S^p\to S^p$ and $\CO:S^p\to S^p$ given in Corollary \ref{cor:bounded}.

\begin{cor}
Let $\Ph$ be an analytic self-map of $\D$ with $\Ph\in S^p$.  Then $\CDO:S^p\to S^p$ is compact if and only if $\CO:S^p\to S^p$ is compact.
\end{cor}

\noindent We also have an analog of Theorem \ref{thm:boundedoperators} for compactness.

\begin{theorem} \label{thm:compactequiv}
Let $\Ph$ be a self-map of $\D$ with $\Ph\in S^p$. The following are equivalent.

\begin{enumerate}
\item $\CDO:S^p\to S^p$ is compact.

\item $\CO:H^p\to S^p$ is compact.

\item $\DCO:H^p\to H^p$ is compact.

\end{enumerate}
\end{theorem}

\begin{proof}
We know that (i) and (iii) are logically equivalent by Theorem \ref{thm:Supnormequivcompact} and \cite[Corollary 1]{HibschweilerPortnoy:2005} since both are equivalent to the condition $\|\Ph\|_{\infty}<1$. It is left to show that (ii) is also equivalent. We first show (i)$\Rightarrow$(ii). Suppose $D_\Ph:S^p\to S^p$ is compact. To show $\CO:H^p\to S^p$ is compact, we will appeal to Proposition \ref{prop:compactreformulation}. To that end, let $(f_n)$ be a bounded sequence in $H^p$ such that $(f_n)\to 0$ uniformly on compact subsets of $\D$.  For $n\in \N$, let $g_n\in S^p$ with $g_n'=f_n$, i.e. take \[g_n(z)=\int_0^z f_n(w)\,dw.\]  Then $(g_n)$ is a bounded sequence in $S^p$ since $\|g_n\|_{S^p}=|g_n(0)|+\|g_n'\|_{H^p}=\|f_n\|_{H^p}$. The hypothesis on $(f_n)$ guarantees that $(g_n)\to 0$ uniformly on compact subsets of $\D$.  Thus $(\|\CDO g_n\|_{S^p})\to 0$ by Proposition \ref{prop:compactreformulation}.  Also, we have \[\|\CO f_n\|_{S^p}=\|\CO g_n'\|_{S^p}=\|g_n'\circ\Ph\|_{S^p}=\|\CDO g_n\|_{S^p}\] and hence $(\|\CO f_n\|_{S^p})\to 0$. Thus $\CO:H^p\to S^p$ is compact by Proposition \ref{prop:compactreformulation}.

To complete the proof, we show (ii)$\Rightarrow$(iii). Suppose $\CO:H^p\to S^p$ is compact and let $(f_n)$ be a bounded sequence in $H^p$ with $(f_n)\to 0$ uniformly on compact subsets of $\D$. Then, since $\CO:H^p\to S^p$ is compact, $(\|\CO f\|_{S^p})\to 0$ by Proposition \ref{prop:compactreformulation}.  But \[\|\DCO f\|_{H^p}=\|(f\circ \Ph)'\|_{H^p}\leq |f(\Ph(0))|+\|(f\circ \Ph)'\|_{H^p}=\|\CO f\|_{S^p}\] and hence $(\|\DCO f\|_{H^p})\to 0$.  Thus $\DCO:H^p\to H^p$ is compact completing the proof.
\end{proof}

\section{Spectrum}
Next we present  two results on the spectrum of $D_{\varphi}$ acting on $S^p$ in the case that $\varphi$ is a special symbol. 

Since a self-map of the disk $\Ph$ with $\|\varphi\|_\infty < 1$ induces a compact operator $D_\varphi$ on $S^p$ by Theorem \ref{thm:Supnormequivcompact}, as it does on the Hardy space $H^2$, the proof of the following theorem follows exactly as presented in \cite[Proposition 3]{FatehiHammond:2020}. 

\begin{theorem}\label{Theorem:SpectrumLinear}
If $\varphi(z) = az+b$, for $0 < |a| < 1-|b|$, then the spectrum of $D_\varphi$ on $S^p$ is $\{0\}$.
\end{theorem} 

We now turn our attention to the class of symbols $\varphi(z)=az^M$ for $0<|a|<1$ and $M\in\mathbb{N}$. We will consider the case for which $M\geq 2$. The  argument for the following theorem is identical to that for the Dirichlet space in \cite{AllenHellerPons:2021} so we present a shortened version. 

\begin{theorem}\label{Theorem:SpectralResults}
If $\varphi(z) = az^M$, for $0 < |a| < 1$ and $M$ in $\N$, then the spectrum of $D_\varphi$ on $S^p$ is given by \[\sigma(D_\varphi) = \begin{cases}\hfil \{0, 2a\} & \text{if $M = 2$}\\
\hfil\{0\} & \text{otherwise.}
\end{cases}
\]
\end{theorem}

\begin{proof}
By the Spectral Theorem for Compact Operators, the spectrum is countable, contains 0, and any nonzero element is an eigenvalue. If we assume the existence of a nonzero eigenvalue $\lambda$ of $D_{\varphi}$, then any associated eigenfunction $f$ satisfies
$$
\frac{d^n}{dz^n}[\lambda f(z)]=\frac{d^n}{dz^n}[f'(\varphi(z))]
$$
for all $z\in\mathbb{D}$ and $n\geq 0$.

Of particular interest is the relation
\begin{align*}
\left.\frac{d^n}{dz^n}[\lambda f(z)]\right\rvert_{z=0} & =\left.\frac{d^n}{dz^n}[f'(\varphi(z))]\right\rvert_{z=0}
\intertext{or}
f^{(n)}(0) &= \frac{1}{\lambda}\left.\frac{d^n}{dz^n}[f'(\varphi(z))]\right\rvert_{z=0}.
\end{align*}


We consider the three different cases $n<M$, $n=M$, and $n>M$ to argue that $f^{(n)}(0)=0$ unless $n=M=2$ and $\lambda=2a$.

Observe that $\varphi(0)=0$ and $\varphi^{(\ell)}(0)=0$ if $\ell\neq M$ so we only need to consider the portion of $\left.\frac{d^n}{dz^n}[f'(\varphi(z))]\right\rvert_{z=0}$ which contains $\varphi^{(M)}(0)$. Note that when $n<M$, $\left.\frac{d^n}{dz^n}[f'(\varphi(z))]\right\rvert_{z=0}=0$ as this expression only contains $\varphi^{(\ell)}(0)$ for $\ell< M$. Therefore $f^{(n)}(0)=0$ whenever $n<M$.

When $n=M$, a quick computation shows that 
$$
\left.\frac{d^n}{dz^n}[f'(\varphi(z))]\right\rvert_{z=0}=f''(0)\varphi^{(M)}(0)
$$
Therefore, 
$$
f^{(M)}(0)=\frac{1}{\lambda}f''(0)\varphi^{(M)}(0)
$$

\noindent If $M=2$, we have
$$
f''(0)=\frac{M! a}{\lambda}f''(0)
$$
and so either $\lambda=M!a$ or $f''(0)=0$. If $M>2$, then $f''(0)=0$ by the previous case and so $f^{(M)}(0)=0$.

Finally suppose $n>M$. Recall we only need to consider the piece of $\left.\frac{d^n}{dz^n}[f'(\varphi(z))]\right\rvert_{z=0}$ which contains $\varphi^{(M)}(0)$, but the $\varphi^{(M)}(0)$ will only appear in products containing $\varphi^{(\ell)}(0)$ for $\ell\neq M$. Therefore again we see that $f^{(n)}(0)=0$ whenever $n>M$. Combining these three cases yields the desired result.
\end{proof}

\section{Norm}\label{Subsection:Norm}

We will consider the norm of the operator $D_\varphi$ acting on $S^2$ when $\varphi(z) = az^M$ for $0 < |a| < 1$ and $M\in\N$. Note that any such $\varphi$ will induce a compact operator on $S^2$ by Theorem \ref{thm:Supnormequivcompact}. Similar results were shown to hold on the Hardy space $H^2$ and on the Dirichlet space $\mathcal{D}$ by \cite{FatehiHammond:2020} and \cite{AllenHellerPons:2021}, respectively. Our argument is similar to that of the Dirichlet space and we include it for completeness.

For $0<|a|<1$ and $M\in\N$, we define constants 
\[\nu = \left\lfloor\frac{2-|a|}{1-|a|}\right\rfloor\]
and 
\[\mathcal{N}_M = \max\left\{1, M(\nu-1)|a|^{\nu-1}\right\}.\]

\begin{theorem} If $\varphi(z) = az^M$ for $0 < |a| < 1$ and $M$ in $\N$, then $\|D_\varphi\| = \mathcal{N}_M$.
\end{theorem}

\begin{proof}
For $\varphi$ as given, we will argue $\|D_\varphi\| \geq \mathcal{N}_M$ and $\|D_\varphi\| \leq \mathcal{N}_M$. To obtain the lower bound we first consider the orthonormal basis functions $(e_n)_{n\geq 0}$ defined by $e_0(z) = 1$ and $e_n(z) = z^n/n$, $n\geq 1$. For $z\in \D$, note $e_0'(\varphi(z))=0$,
\[e_1'(\varphi(z)) = 1 = e_0(z),\]
and, for $n \geq 2$,
\[e_n'(\varphi(z)) =  a^{n-1}z^{M(n-1)} =  M(n-1)a^{n-1}e_{M(n-1)}(z).\] Thus \[\|D_\varphi\| \geq \max\left\{1,\sup_{n \geq 2}M(n-1)|a|^{n-1}\right\}.\] To verify the lower bound we must show \[\sup_{n \geq 2}M(n-1)|a|^{n-1} = M(\nu-1)|a|^{\nu-1}.\] Note that the function $g(x) = M(x-1)|a|^{x-1}$ has exactly one critical point in $(1,\infty)$. This point is a local maximum and hence also the absolute maximum of $g$ on $(1,\infty)$.  Thus the supremum we wish to compute is a maximum and will occur at the greatest integer $n \geq 2$ such that 
\[(n-2)|a|^{n-2} \leq (n-1)|a|^{n-1};\] equivalently,
\[n \leq \frac{2-|a|}{1-|a|}.\]
Thus, \begin{equation}\label{eqn:supremum}\sup_{n \geq 2}M(n-1)|a|^{n-1} = M(\nu-1)|a|^{\nu-1},\end{equation} which implies $\|D_\varphi\| \geq \mathcal{N}_M$. 

For the upper estimate, let $f(z) = \sum_{n=0}^{\infty} b_nz^n= b_0e_0(z)+\sum_{n=1}^{\infty}b_nne_n(z)$ be in $S^2$. Consider
\[\begin{aligned}(D_\varphi f)(z)&= \sum_{n=1}^{\infty} b_nne_n'(\varphi(z))\\
&=b_1+\sum_{n=2}^{\infty}b_nn\left(M(n-1)a^{n-1}e_{M(n-1)}(z)\right).
\end{aligned}\] With this and \myeqref{eqn:supremum}, we have  
\[\begin{aligned}
\|D_\varphi f\|_{S^2}^2 &= |b_1|^2+\sum_{n=2}^{\infty}|b_n|^2n^2\left(M(n-1)|a|^{n-1}\right)^2\\
&\leq |b_1|^2+\sum_{n=2}^{\infty}|b_n|^2n^2\left(M(\nu-1)|a|^{\nu-1}\right)^2\\
&\leq \mathcal{N}_M^2\left(|b_1|^2+\sum_{n=2}^{\infty}|b_n|^2n^2\right) \\
&\leq \mathcal{N}_M^2 \|f\|_{S^2}^2\\
\end{aligned}\] and therefore $\|D_\varphi\|\leq \mathcal{N}_M$ as desired.
\end{proof}

Considering the quantities in the theorem, we can understand how the norm changes with respect to $|a|$.  When $M=1$, $\|D_\varphi\|=1$ when $0<|a|\leq3^{-1/3}$ and $\|D_\varphi\|>1$ for $3^{-1/3}<|a|<1$. For $M\geq 2$, $\|D_\varphi\|=1$ when $0<|a|\leq1/M$ and $\|D_\varphi\|>1$ for $1/M<|a|<1$. In both cases, the norm tends to infinity as $|a|\to1$. 

\begin{figure}[H]
\begin{center}
\resizebox {.9\textwidth} {\height} {
\begin{tikzpicture}[baseline]
    \begin{axis}[
	    axis x line=middle,
		axis y line=middle,
		axis on top,
		width=.75\textwidth,
		ylabel={\footnotesize $\|D_\varphi\|$},
		xlabel={\footnotesize $|a|$},
		ytick={1,3},
		yticklabels={1,3},
		xtick={0,0.33333,.5,0.69336127,1},
		xticklabels={0,$\frac{1}{3}$,$\frac{1}{2}$,$\frac{1}{\sqrt[3]{3}}$,$1$},
		xtick align=outside,
		xmin=0, xmax=1.05, ymin=0, ymax=4.25,
		]
		\addplot[smooth,black,ultra thick,domain=0:0.69336127] {1};
		\addplot[smooth,black,thick,domain=0.69336127:0.95] {floor((2-x)/(1-x))-1)*x^(floor((2-x)/(1-x))-1)};
		\addplot[smooth,black,thick,domain=0.5:0.95] {2*(floor((2-x)/(1-x))-1)*x^(floor((2-x)/(1-x))-1)};
		\addplot[smooth,black,thick,domain=0.33333:0.95] {3*(floor((2-x)/(1-x))-1)*x^(floor((2-x)/(1-x))-1)};
		\node at (axis cs:0.69336,1) [anchor=north] {\scriptsize $M$=$1$};
		\node at (axis cs:.5,1) [anchor=north] {\scriptsize $M$=$2$};
		\node at (axis cs:0.33333,1) [anchor=north] {\scriptsize $M$=$3$};
	\end{axis}
\end{tikzpicture}
}
\caption{Norm of $D_\varphi$ for $\varphi(z) = az^M$, as a function of $|a|$, with $M=1,2,3$.}\label{Figure:NormDphi}
\end{center}
\end{figure}
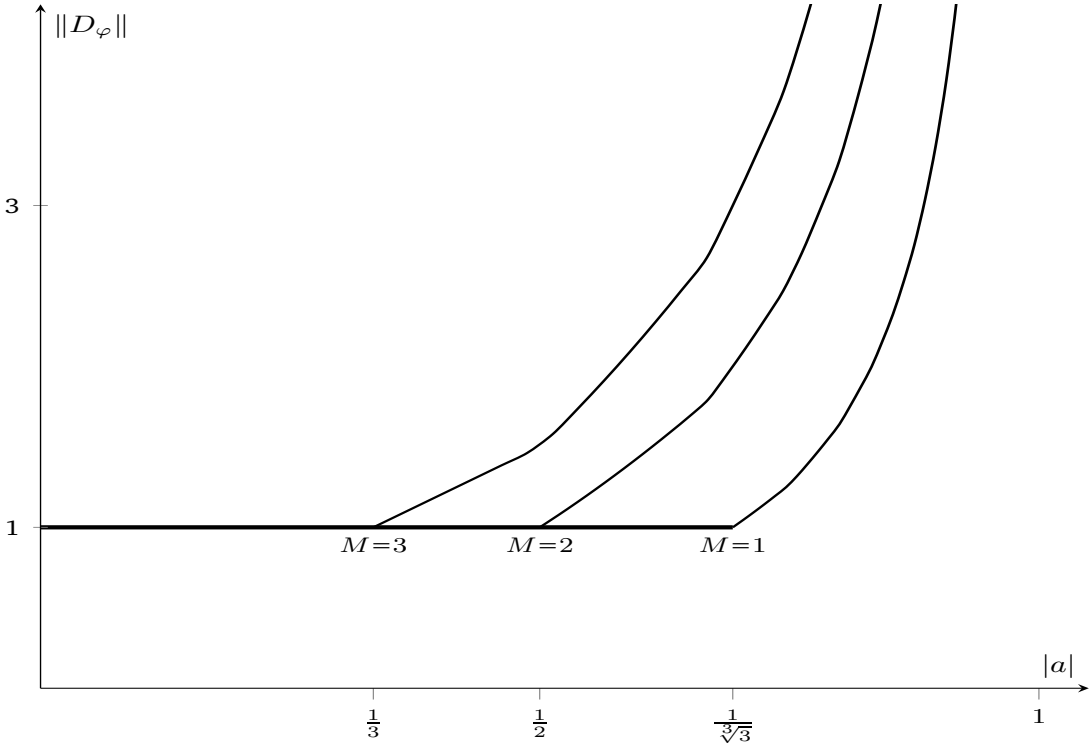

\section{Adjoints}

Finally, in this section, we consider the adjoint of a composition-differentiation operator on several spaces. First recall the definition of the standard weighted Bergman space: for $\alpha >-1$, $A_{\alpha}^2$ is given by \[\left\{f \textup{ analytic in } \D: \|f\|_{\alpha}^2=(\alpha+1)\int_\D|f(z)|^2(1-|z|^2)^{\alpha}\,dA(z)<\infty\right\},\] where $dA$ is normalized Lebesgue area measure. The reproducing kernels for $A_{\alpha}^2$ are given by 
\[
K_w(z)=\frac{1}{(1-\overline{w}z)^{\alpha +2}}
\]
while the kernels for evaluation of the first derivative are given by 
\begin{equation}\label{eqn:kerderbergman}
K_w^{(1)}(z)=\frac{(\alpha+2)z}{(1-\overline{w}z)^{\alpha+3}}.
\end{equation}

Let $\varphi(z)=(az+b)/(cz+d)$ be a linear fractional self-map of the unit disk and let $\sigma(z)=(\overline{a}z-\overline{c})/(-\overline{b}z+\overline{d})$ be the Krein adjoint of $\varphi$. It is known that $\sigma$ is also a self-map of the disk. Throughout this section we will assume $\|\varphi\|_{\infty}<1$ so that $\CDO$ is bounded (even compact) on $A_{\alpha}^2$; see \cite{HibschweilerPortnoy:2005}.

For an analytic function $\psi:\mathbb{D}\rightarrow\mathbb{C}$, we define the Toeplitz operator $T_{\psi}$ by $T_{\psi}(f)=\psi \cdot f$ for $f$ in some functional Hilbert space of analytic functions. It is well known that $T_{\psi}$ is bounded on any weighted Bergman space if and only if $\psi$ is in $H^{\infty}$ and any self-map of $\D$ will induce a bounded composition operator on a weighted Bergman space. It is also well known that $C_{\varphi}^*(K_w)=K_{\varphi(w)}$ and $T_{\psi}^*(K_w)=\overline{\psi(w)}K_w$ for any $w\in\mathbb{D}$.
Fatehi and Hammond showed that $\CDO^*(K_w(z))=K^{(1)}_{\varphi(w)}(z)$ (see \cite{FatehiHammond:2020}) whenever $\CDO$ is bounded on $H^2$ and this behavior persists on other Hilbert spaces.

We first show that the adjoint formula $D_{\varphi}^*T_{K^{(1)}_{\sigma(0)}}^*=T_{K^{(1)}_{\varphi(0)}}D_{\sigma}$ previously shown to hold on the Hardy space $H^2$ \cite[Theorem 1]{FatehiHammond:2020} in fact holds on any weighted Bergman space $A^2_{\alpha}$ for $\alpha>-1$. For the Hardy space, the kernels appearing as multipliers are \begin{equation}\label{eqn:kerderhardy}K_w^{(1)}=\frac{z}{(1-\overline{w}z)^2}.\end{equation}

\begin{theorem} \label{thm:adjointBergman}
Let $\varphi$ be a linear fractional self-map of $\mathbb{D}$ with $\|\Ph\|<1$. If $\alpha>-1$, then on $A^2_{\alpha}$, $D_{\varphi}^*T_{K^{(1)}_{\sigma(0)}}^*=T_{K^{(1)}_{\varphi(0)}}D_{\sigma}$.
\end{theorem}

\begin{proof}
Using the kernel functions for $A_{\alpha}^2$, we have
\begin{eqnarray*}
D_{\varphi}^*T_{K^{(1)}_{\sigma(0)}}^*(K_w)(z)&=&\overline{K^{(1)}_{\sigma(0)}(w)}K^{(1)}_{\varphi(w)}(z)\\
&=& \frac{(\alpha+2)\overline{w}}{(1+\overline{(c/d)w)})^{\alpha+3}}\frac{(\alpha+2)z}{(1-\overline{\varphi(w)}z)^{\alpha+3}}\\
&=& \frac{(\alpha+2)\overline{w}\overline{d}^{\alpha+3}}{(\overline{cw+d})^{\alpha+3}}\frac{(\alpha+2)z}{(1-\overline{\varphi(w)}z)^{\alpha+3}}
\end{eqnarray*}
for any $w$ in $\mathbb{D}$ while
\begin{eqnarray*}
T_{K^{(1)}_{\varphi(0)}}D_{\sigma}(K_w)(z)&=&K^{(1)}_{\varphi(0)}(z)K'_w(\sigma(z))\\
&=&\frac{(\alpha+2)z}{(1-(\overline{b/d})z)^{\alpha+3}}\frac{(\alpha+2)\overline{w}}{(1-\overline{w}\sigma(z))^{\alpha+3}}\\
&=&\frac{(\alpha+2)z\overline{d}^{\alpha+3}}{(-\overline{b}z+\overline{d})^{\alpha+3}}\frac{(\alpha+2)\overline{w}}{\left(1-\overline{w}\left(\frac{\overline{a}z-\overline{c}}{-\overline{b}z+\overline{d}}\right)\right)^{\alpha+3}}\\
&=&\frac{(\alpha+2)^2\overline{d}^{\alpha+3}\overline{w}z}{(-\overline{b}z+\overline{d}-\overline{aw}z+\overline{cw})^{\alpha+3}}\\
&=&\frac{(\alpha+2)^2\overline{d}^{\alpha+3}\overline{w}z}{(\overline{cw+d})^{\alpha+3}(1-\overline{\varphi(w)}z)^{\alpha+3}}.
\end{eqnarray*}
Therefore $D_{\varphi}^*T_{K^{(1)}_{\sigma(0)}}^*=T_{K^{(1)}_{\varphi(0)}}D_{\sigma}$ on the span of the reproducing kernel functions, which constitutes a dense subset of $A^2_{\alpha}$. Hence the two operators are equal. 
\end{proof}

\noindent As we will see below when we move to the space $S^2$, we see that a seemingly different adjoint formula emerges, but there are some new challenges.  Reproducing kernel functions are an important tool and have been particularly useful in studying the adjoint for a composition operator. Because no nice closed form exists for the reproducing kernel on $S^2$ given in Section 2, we will instead consider $S^2$ as a weighted Hardy space with an equivalent norm, which we will denote $\widetilde{S}^2$. Define the weight sequence $\beta(0)=\beta(1)=1$, $\beta(n)=\sqrt{n(n-1)}$ for $n\geq 2$ so that for $f(z)=\sum_{n=0}^{\infty}a_n z^n$ and $g(z)=\sum_{n=0}^{\infty}b_n z^n$, we have
$$
\langle f,g\rangle=a_0\overline{b_0}+a_1\overline{b_1}+\sum_{n=2}^{\infty} a_n\overline{b_n}n(n-1)
$$
With this weight sequence we see that 
\begin{eqnarray*}
K_w(z)&=&1+\overline{w}z+\sum_{n=2}^{\infty} \frac{(\overline{w}z)^n}{n(n-1)}\\
&=&1+2\overline{w}z-\overline{w}z+\sum_{n=2}^{\infty} \frac{(\overline{w}z)^n}{n(n-1)}\\
&=&1+2\overline{w}z+(1-\overline{w}z)\log(1-\overline{w}z).
\end{eqnarray*}
and it follows that on $\widetilde{S}^2$
\begin{eqnarray*}
K^{(1)}_w(z)&=&z-z\log(1-\overline{w}z)
\end{eqnarray*}

\bigskip

For $s=1,2$, let $H_s$ be a  Hilbert space of analytic functions on the disk such that
$$
\langle z^j, z^\ell\rangle_s = \begin{cases}
0 & \text{ if } j\neq \ell \\
\beta_s^2(j) & \text{ if } j=\ell
\end{cases}
$$
where $\{\beta_s(n)\}_{n=0}^{\infty}$ is a sequence of positive real numbers with $\liminf{\beta_s(n)^{1/n}}=1$. We also assume there is a constant $\alpha>0$ such that $\lim_{n\to \infty}\beta_2(n)/\beta_1(n)=\alpha$ so that the norms on $H_1$ and $H_2$ are equivalent.

\begin{prop}[{\cite[Proposition 3.1]{CuckovicLe:2016}}]\label{AdjointCompactDiff} Let $A$ be a bounded linear operator on $H_1$ (hence $A$ is also bounded on $H_2$). Let $B_s$ be the adjoint of $A$ on $H_s$ for $s=1,2$. Then $B_2-B_1$ is a compact operator on $H_2$ (hence, on $H_1$ as well).
\end{prop}

\noindent Note that by the above proposition, it suffices to prove our adjoint formula for the space $\widetilde{S}^2$, which is norm equivalent to $S^2$.

For $\varphi(z)=(az+b)/(cz+d)$, define $\eta(z)=(cz+d)^{-1}$ and $\mu(z)=-\overline{b}z+\overline{d}$, which are bounded analytic functions on a neighborhood of the closed unit disk with
$$
1-\overline{\varphi(w)}z=\mu(z)(1-\overline{w}\sigma(z))\overline{\eta(w)}.
$$
Consequently by choosing appropriate branches of the logarithms, we have
\begin{equation}\label{adjointswitch}
\log(1-\overline{\varphi(w)}z)=\log(\mu(z))+\log(1-\overline{w}\sigma(z))+\log(\overline{\eta(w)}),
\end{equation}
(see \cite[Eq.~(2.8)]{CuckovicLe:2016}). 

\begin{lemma}\label{finiterank}
Let $\mathcal{H}$ be a Hilbert space of analytic functions on the unit disc with reproducing kernel $K_w$. Let $\mathcal{X}$ be the set of functions on $\mathbb{D}\times\mathbb{D}$ of the form
$$
f_1(z)\overline{g_1(w)}+\cdots+f_m(z)\overline{g_m(w)}
$$
where $f_1,...,f_m$ and $g_1,...,g_m$ belong to $\mathcal{H}$ and $m$ is a positive integer. Then a bounded linear operator $A:\mathcal{H}\rightarrow\mathcal{H}$ has finite rank if and only if the function $(z,w)\rightarrow\langle AK_w,K_z\rangle$ belongs to $\mathcal{X}$.
\end{lemma}

\begin{lemma}\label{lemma:boundedmultiplicationoperator}
Let $\psi$ be analytic in $\D$ with both $\psi$ and $\psi'$ in $A(\D)$. Then $T_{\psi}$ is a bounded multiplication operator on $S^p$.
\end{lemma}

\noindent The lemma guarantees us that multiplication by the function $g(z)=z$ induces a bounded operator on $S^2$.  

\begin{theorem}\label{thm:S2adjoint}
Let $\Ph$ be a linear fractional self-map of $\D$ with $\|\Ph\|_{\infty}<1$ and consider $\CDO:\widetilde{S}^2\to \widetilde{S}^2$. Then there is a finite rank operator $K_1$ such that $D_{\Ph}^*T_{z}^*=T_{z}D_{\SI} +K_1$.
\end{theorem}

\begin{proof}
Suppose $\varphi(z)=(az+b)/(cz+d)$ be a linear fractional self-map of the disc with $\|\varphi\|_{\infty}<1$ and $\sigma(z)=(\overline{a}z-\overline{c})/(-\overline{b}z+\overline{d})$ be the Krein adjoint of $\varphi$.  Define $\eta(z)=(cz+d)^{-1}$ and $\mu(z)=-\overline{b}z+\overline{d}$. First observe that
\begin{eqnarray*}
D_{\varphi}^*T_z^*(K_w)(z)&=&\overline{w}K^{(1)}_{\varphi(w)}(z)\\
&=&\overline{w}(z-z\log(1-\overline{\varphi(w)}z))\\
&=&\overline{w}z(1-\log(1-\overline{\varphi(w)}z))
\end{eqnarray*}
and
\begin{eqnarray*}
T_zD_{\sigma}(K_w)(z)&=&zK'_w(\sigma(z))\\
&=&z\overline{w}(1-\log(1-\overline{w}\sigma(z))).
\end{eqnarray*}
Subtracting the two and using \myeqref{adjointswitch} yields
\begin{eqnarray*}
(D_{\varphi}^*T_z^*-T_zD_{\sigma})(K_w)(z)&=&\overline{w}z(-\log(1-\overline{\varphi(w)}z)+\log(1-\overline{w}\sigma(z)))\\
&=&\overline{w}z(-\log(\mu(z))-\log(1-\overline{w}\sigma(z))-\log(\overline{\eta(w)})+\log(1-\overline{w}\sigma(z)))\\
&=&-\overline{w}z(\log(\mu(z))+\log(\overline{\eta(w)})),
\end{eqnarray*}
which is a finite rank operator by Lemma \ref{finiterank}.
\end{proof}

\noindent Combining the above result with Proposition \ref{AdjointCompactDiff} yields the immediate Corollary.

\begin{cor}\label{cor:S2adjoint}
Let $\Ph$ be a linear fractional self-map of $\D$ with $\|\Ph\|_{\infty}<1$ and consider $\CDO:S^2\to S^2$. Then there is a compact operator $K_2$ such that $D_{\Ph}^*T_{z}^*=T_{z}D_{\SI} +K_2$.
\end{cor}

To understand why this adjoint formula does not seem to resemble that for the weighted Bergman spaces given earlier, in particular why the multipliers are not kernel functions for evaluation of the first derivative at $\varphi(0)$ and $\sigma(0)$, we need a broader perspective. For $\alpha \in \mathbb{R}$, we say that a function $f$ (analytic in $\D$) is in $A_{\alpha}^2$ if there is a non-negative integer $k$ such that $\alpha+2k>-1$ and $f^{(k)}$ belongs to $A_{\alpha+2k}^2$. For $\alpha >-1$, this definition is equivalent to our definition for the standard weighted Bergman spaces given at the beginning of this section. See \cite{ZhaoZhu2008} for more information. For $\alpha\in \mathbb{R}$ it is known that $A_{\alpha}^2$ is a reproducing kernel Hilbert space with an appropriate norm/inner product and we have the following form for the reproducing kernels: for $\alpha+2>0$, 
\[
K_w^{\alpha}(z)=\frac{1}{(1-\overline{w}z)^{\alpha+2}};
\] when $-N<\alpha+2<-N+1$ for $N$ a positive integer, we have
\[
K_w^{\alpha}(z)=\frac{(-1)^N}{(1-\overline{w}z)^{\alpha+2}}+Q(\overline{w}z),
\] where $Q$ is an analytic polynomial of degree $N$; and for $\alpha+2=-N$ for $N$ a non-negative integer, then
\[
K_w^{\alpha}(z)=(\overline{w}z-1)^N\log\left(\frac{1}{1-\overline{w}z}\right)+Q(\overline{w}z), 
\] where $Q$ is an analytic polynomial of degree $N$. Moreover, for $\alpha\in\mathbb{R}$, $\|z^n\|_{\alpha}$ is asymptotic to $n^{-(\alpha+1)/2}$ as $n\to\infty$. Thus $A_{\alpha}^2$ is norm equivalent to the weighted Hardy space with weight sequence defined by $\beta(0)=1$ and $\beta(n)=n^{-(\alpha+1)/2}$. For $\alpha>-1$, this space is norm equivalent to $A_{\alpha}^2$. When $\alpha=-1,-2, -3$ the space is norm equivalent to the Hardy space, the Dirichlet space, and $S^2$, respectively.

Our proof of Theorem \ref{thm:adjointBergman}, shows that a similar adjoint formula would hold for $\CDO$ acting on a weighted Bergman space with $\alpha>-2$ since the kernels have a similar form in this case; this includes the Hardy space. In \cite{AllenHellerPons:2021}, the authors show that for $\CDO$ acting on the Dirichlet space ($\alpha=-2$) and $\Ph$ linear fractional with $\|\Ph\|_{\infty}<1$, the formula $D_{\varphi}^*T_{K^{(1)}_{\sigma(0)}}^*=T_{K^{(1)}_{\varphi(0)}}D_{\sigma}$ holds where in this case we have \begin{equation}\label{eqn:kerderdirichlet}K_w^{(1)}=\frac{z}{1-\overline{w}z}.\end{equation}

Now we can tie all of these formulas together. For $\alpha \in \mathbb{R}$, let \[g_{\alpha}(z)=\frac{z}{(1-\overline{\Ph(0)}z)^{\alpha+3}} \textrm{\hspace{.2in}and\hspace{.2in}} h_{\alpha}(z)=\frac{z}{(1-\overline{\sigma(0)}z)^{\alpha+3}}.\] Note that for $\alpha=-2$, $g_{\alpha}(z)=K^{(1)}_{\varphi(0)}(z)$ and $h_{\alpha}(z)=K^{(1)}_{\sigma(0)}(0)$ for the kernels in Eqn. (\ref{eqn:kerderdirichlet}). Similarly, when $\alpha=-1$, these same relationships hold for the kernels from Eqn. (\ref{eqn:kerderhardy}). When $\alpha>-1$, the kernels in Eqn. (\ref{eqn:kerderbergman}) have a factor of $\alpha+2$, but these would cancel in the adjoint calculation.  Thus, combining the results from \cite{AllenHellerPons:2021} and \cite{FatehiHammond:2020} with our Theorem \ref{thm:adjointBergman} and Corollary \ref{cor:S2adjoint}, we have the following result.

\begin{theorem}
Let $\varphi$ be a linear fractional self-map of $\mathbb{D}$ with $\|\Ph\|_{\infty}<1$ and consider $\CDO:A_{\alpha}^2\to A_{\alpha}^2$. If $\alpha\geq-2$ or $\alpha=-3$, then there is a compact operator $K$ such that $D_{\Ph}^*T_{h_{\alpha}}^*=T_{g_{\alpha}}D_{\SI} +K$.
\end{theorem}

\noindent It seems likely that this result holds for any $\alpha\in\mathbb{R}$ following the methods of \cite{CuckovicLe:2016}.

\bibliographystyle{amsplain}
\bibliography{references.bib}

\providecommand{\bysame}{\leavevmode\hbox to3em{\hrulefill}\thinspace}
\providecommand{\MR}{\relax\ifhmode\unskip\space\fi MR }
\providecommand{\MRhref}[2]{%
  \href{http://www.ams.org/mathscinet-getitem?mr=#1}{#2}
}
\providecommand{\href}[2]{#2}
\begin{thebibliography}{10}

\bibitem{AllenHellerPons:2015}
Robert~F. Allen, Katherine~C. Heller, and Matthew~A. Pons, \emph{Multiplication
  operators on {$S^2(\mathbb{D})$}}, Acta Sci. Math. (Szeged) \textbf{81}
  (2015), no.~3-4, 575--587. \MR{3443772}

\bibitem{AllenHellerPons:2021}
\bysame, \emph{Composition-differentiation operators on the {D}irichlet space},
  J. Math. Anal. Appl. \textbf{512} (2022), no.~2, Paper No. 126186, 18.
  \MR{4398431}

\bibitem{CowenMacCluer:1995}
Carl~C. Cowen and Barbara~D. MacCluer, \emph{Composition operators on spaces of
  analytic functions}, Studies in Advanced Mathematics, CRC Press, Boca Raton,
  FL, 1995. \MR{1397026}

\bibitem{CuckovicLe:2016}
Zeljko Cuckovic and Trieu Le, \emph{Adjoints of linear fractional composition
  operators on weighted hardy spaces}, Acta Sci. Math. (Szeged) \textbf{82}
  (2016), 651--662.

\bibitem{FatehiHammond:2020}
Mahsa Fatehi and Christopher N.~B. Hammond, \emph{Composition-differentiation
  operators on the {H}ardy space}, Proc. Amer. Math. Soc. \textbf{148} (2020),
  no.~7, 2893--2900. \MR{4099777}

\bibitem{HibschweilerPortnoy:2005}
R.~A. Hibschweiler and N.~Portnoy, \emph{Composition followed by
  differentiation between {B}ergman and {H}ardy spaces}, Rocky Mountain J.
  Math. \textbf{35} (2005), no.~3, 843--855. \MR{2150311}

\bibitem{Lang:1993}
Serge Lang, \emph{Complex analysis}, third ed., Graduate Texts in Mathematics,
  vol. 103, Springer-Verlag, New York, 1993. \MR{1199813}

\bibitem{Ohno:2006}
Sh\^{u}ichi Ohno, \emph{Products of composition and differentiation between
  {H}ardy spaces}, Bull. Austral. Math. Soc. \textbf{73} (2006), no.~2,
  235--243. \MR{2217942}

\bibitem{Ohno:2009}
\bysame, \emph{Products of differentiation and composition on {B}loch spaces},
  Bull. Korean Math. Soc. \textbf{46} (2009), no.~6, 1135--1140. \MR{2583478}

\bibitem{Ohno:2014}
\bysame, \emph{Composition operators related to the {D}irichlet space}, Bull.
  Belg. Math. Soc. Simon Stevin \textbf{21} (2014), no.~4, 759--767.
  \MR{3271331}

\bibitem{Shapiro}
Joel Shapiro, \emph{Compact composition operators on spaces of boundary-regular
  holomorphic functions}, Proc. Amer. Math. Soc. \textbf{100} (1987), 49--57.

\bibitem{Tjani:2003}
Maria Tjani, \emph{Compact composition operators on {B}esov spaces}, Trans.
  Amer. Math. Soc. \textbf{355} (2003), no.~11, 4683--4698. \MR{1990767}

\bibitem{ZhaoZhu2008}
Ruhan Zhao and Kehe Zhu, \emph{Theory of {B}ergman spaces in the unit ball of
  $\mathbb{C}^n$}, M\'{e}m. Soc. Math. Fr. (N.S.) (2008), no.~115, vi+103 pp.
  (2009). \MR{2537698}

\end{thebibliography}

\end{document}